\documentclass[11pt, a4paper, twoside,reqno]{amsart}
\usepackage[centering, totalwidth = 390pt, totalheight = 600pt]{geometry}
\usepackage{amssymb, amsmath, amsthm,scalerel}
\usepackage{microtype, stmaryrd, url, lmodern, eucal}
\usepackage[shortlabels]{enumitem}
\usepackage[latin1]{inputenc}
\usepackage{color}
\definecolor{darkgreen}{rgb}{0,0.45,0}
\usepackage[colorlinks,citecolor=darkgreen,linkcolor=darkgreen]{hyperref}
 \usepackage{tikz,tikz-cd}
 \usepackage{xspace}

\numberwithin{equation}{section}

\theoremstyle{plain}

\newtheorem{Theorem}{Theorem}
\newtheorem{thm}[Theorem]{Theorem}

\newtheorem{Prop}[Theorem]{Proposition}
\newtheorem{Lemma}[Theorem]{Lemma}

\theoremstyle{definition}
\newtheorem{Definition}[Theorem]{Definition}

\newtheorem{Example}[Theorem]{Example}
\newtheorem{Remark}[Theorem]{Remark}

\newtheorem{Properties}[Theorem]{Properties}

\newtheorem{defn}[Theorem]{Definition}
\newtheorem{example}[Theorem]{Example}

\newtheorem*{notation}{Notation}

\newtheorem{Assumptions}[Theorem]{Assumptions}

\input xy
\xyoption{all}
\SelectTips{cm}{11}
\makeatletter
\def\matrixobject@{%
 \edef \next@{={\DirectionfromtheDirection@ }}%
 \expandafter \toks@ \next@ \plainxy@
 \let\xy@@ix@=\xyq@@toksix@
 \xyFN@ \OBJECT@}
\let\xy@entry@@norm=\entry@@norm
\def\entry@@norm@patched{%
 \let\object@=\matrixobject@
 \xy@entry@@norm }
\AtBeginDocument{\let\entry@@norm\entry@@norm@patched}
\makeatother

\newcommand{\wfscat}{pre-model category}
\newcommand{\wfscats}{pre-model categories}

\newcommand{\C}{{\mathcal C}}
\newcommand{\Ccal}{\mathcal{C}}
\newcommand{\Acal}{\mathcal{A}}
\newcommand{\Vcal}{\mathcal{V}}

\newcommand{\Cat}{\mathbf{Cat}}
\newcommand{\Set}{\mathbf{Set}}
\newcommand{\Gph}{\mathbf{Gph}}

\newcommand{\alg}[1]{\boldsymbol{#1}}
\newcommand{\cat}[1]{\mathbf{#1}}
\newcommand{\Coalg}[1]{\mathsf{#1}\text-\cat{Coalg}}
\newcommand{\Alg}[1]{\mathsf{#1}\text-\cat{Alg}}

\newcommand{\atwo}{{\mathbf 2}}
\newcommand{\athree}{{\mathbf 3}}
\newcommand{\awfs}{\textsc{awfs}\xspace}
\newcommand{\JFib}{\cat{J}\text{-}\cat{Fib}}

\title{Algebraically cofibrant and fibrant objects revisited}

\author{John Bourke}
\address{Department of Mathematics and Statistics, Masaryk University, Kotl\'a\v rsk\'a 2, Brno 61137, Czech Republic}
\email{bourkej@math.muni.cz}
\author{Simon Henry} 
\address{Department of Mathematics and Statistics, University of Ottawa, STEM complex, 150 Louis-Pasteur, K1N 6N5, Ottawa, ON, Canada}\email{shenry2@uottawa.ca}
\thanks{The first named author acknowledges the support of the Grant Agency of the Czech Republic under the grant 19-00902S}
\begin{document}
%\date{\today}

%\subjclass[2020]{}

%\keywords{Primary: 55U35; Secondary 18G30, 18E35}

%It seems that \subjclass[2020] is not implemented yet, in the meantime:

\renewcommand{\thefootnote}{\fnsymbol{footnote}} 
\footnotetext{\emph{Keywords.} Algebraically cofibrant and fibrant objects, weak model categories. }
%\footnotetext{\emph{2020 Mathematics Subject Classification.} Primary: 55U35; Secondary 18G30, 18E35 }
\footnotetext{\emph{2020 Mathematics Subject Classification.} Primary: 55U35; Secondary 18N40, 18C35.}
\renewcommand{\thefootnote}{\arabic{footnote}} 

%\thanks{}
 
 \leftmargini=2em

\def\xypic{\hbox{\rm\Xy-pic}}
\maketitle

\begin{abstract}

We extend all known results about transferred model structures on algebraically cofibrant and fibrant objects by working with weak model categories. We show that for an accessible weak model category there are always Quillen equivalent transferred weak model structures on both the categories of algebraically cofibrant and algebraically fibrant objects.  Under additional assumptions, these transferred weak model structures are shown to be left, right or Quillen model structures.  By combining both constructions, we show that each combinatorial weak model category is connected, via a chain of Quillen equivalences, to a combinatorial Quillen model category in which all objects are fibrant.

\end{abstract}

\section{Introduction}

Given a combinatorial Quillen model category $\Ccal$, one can construct a Quillen equivalent combinatorial model category in which all objects are fibrant \cite{Nikolaus2011Algebraic, Bourke2019Equipping}. This is achieved by considering the category $\Alg{T}$ of \emph{algebraically fibrant} objects in $\Ccal$, which contains objects of $\Ccal$ equipped with chosen lifts against all generating trivial cofibrations and morphisms preserving the chosen lifts.  The category $\Alg{T}$ is itself locally presentable and comes equipped with an adjunction:

\[ F : \Ccal \leftrightarrows \Alg{T} : U\]

where $U$ is the forgetful functor and $F$ its left adjoint --- the Quillen equivalent model structure on $\Alg{T}$ is obtained by right transfer along $U$. In the original paper of Nikolaus \cite{Nikolaus2011Algebraic}, this was done under the mild assumption that every trivial cofibration in $\Ccal$ is a monomorphism, but this assumption was later removed by the first named author in \cite{Bourke2019Equipping}.

A dual version of this result was developed by Ching and Riehl in \cite{ching2014coalgebraic}.  This uses the notion of algebraic weak factorisation system, which endows the cofibrant replacement construction with the structure of a comonad $Q$, whose coalgebras are the \emph{algebraically cofibrant} objects.  (The $T$ of the previous example is the corresponding fibrant replacement monad.)  Again, $\Coalg{Q}$ is locally presentable and comes equipped with an adjunction

\[ U: \Coalg{Q} \leftrightarrows \Ccal : G \]

where $U$ is the forgetful functor and $G$ its right adjoint, and one can try to left transfer the model structure along $U$. Ching and Riehl showed that if $\Ccal$ is a combinatorial \emph{simplicial} model category, and if $Q$ is the simplicial comonad obtained by running the \emph{enriched} algebraic small object argument, then the left transferred model structure exists and is Quillen equivalent to that on $\Ccal$. 

There are, however, some asymmetries in the above.  Firstly, there exist combinatorial model structures for which $\Coalg{Q}$ does not admit the left-transferred model structure --- a simple example was described to us by Alexander Campbell --- whereas $\Alg{T}$ always admits the right transferred model structure.  We will see, however, that $\Coalg{Q}$ is always a right semi-model category.  On the other hand, even if $\Ccal$ is only a combinatorial right semi-model category we will see that $\Alg{T}$ is a genuine Quillen model category. This result captures, for instance, algebraic Kan complexes on \emph{semi-simplicial sets} as well as Nikolaus' guiding example of algebraic Kan complexes on simplicial sets.  

One of our goals in the present paper is to understand the dualisable aspects of the above results.  Evidently, this should include results about semi-model categories.  We will work using the framework of \emph{weak model categories}, recently introduced by the second named author \cite{henry2018weakmodel, henry2019CWMS}, which includes both left and right semi-model categories.  Roughly, weak model categories are like Quillen model categories except that some properties fail to hold when one considers maps that are not from a cofibrant to a fibrant object.  However, as these are the only maps that are really homotopically meaningful in a model category, this weakening does not affect the homotopy theoretic properties of model categories in an essential way.  

Two of our main results, Theorem~\ref{th:ModelStr_Cofibrant_object} of Section~\ref{sect:cofibrant} and Theorem~\ref{th:ModelStr_fibrant_object} of Section~\ref{sect:fibrant}, are dual. They respectively assert that an accessible weak model structure on $\Ccal$ always left transfers to $\Coalg{Q}$ and right transfers to $\Alg{T}$.  If the base $\Ccal$ is core left saturated --- for instance, if $\Ccal$ is a left or right semi-model category --- the weak model structure on $\Coalg{Q}$ is always a right semi-model category, with the dual applying to $\Alg{T}$.  Finally, we identify, in Theorems'~\ref{thm:create} and ~\ref{thm:create2}, stronger conditions --- concerning lifting of cylinder and path objects respectively --- that ensure $\Coalg{Q}$ and $\Alg{T}$ are genuine model categories.  Using these results, we recover and extend the known results on model structures on algebraically cofibrant objects \cite{ching2014coalgebraic} and on algebraically fibrant objects \cite{Nikolaus2011Algebraic, Bourke2019Equipping}.

In Section~\ref{sect:strictify} we apply the above to obtain the following rigidification results for weak model categories.
\begin{itemize}
\item Each accessible weak model category is connected, via a zigzag of Quillen equivalences, to an accessible right semi-model category in which all objects are cofibrant.
\item Similarly, each accessible weak model category is connected, via a zigzag of Quillen equivalences, to an accessible left semi-model category in which all objects are fibrant.
\item Moreover, each \emph{combinatorial} weak model category is connected, via a zigzag of Quillen equivalences, to a combinatorial Quillen model category in which all objects are fibrant.
\end{itemize}

Finally, let us mention that a key tool in the proof of all of these results are new transfer theorems --- see Theorem~\ref{thm:main_transfert} and the dual Theorem~\ref{thm:main_dual_transfert} --- of independent interest.  These give necessary and sufficient conditions for the existence of a transferred weak model structure Quillen equivalent to the original one.

\section{Preliminaries on weak model categories}\label{sect:prelim}

In the present section we review the necessary background on weak model categories from \cite{henry2018weakmodel} and \cite{henry2019CWMS}.

\subsection{Weak model categories}

\begin{defn}
A \wfscat\ is a complete and cocomplete category equipped with two weak factorisation systems respectively called (cofibration/trivial fibration) and (trivial cofibration/fibration) such that each trivial fibration is a fibration, or equivalently, such that each trivial cofibration is a cofibration.
\end{defn}
\begin{notation}
In a \wfscat\ an important role is also played by the core and acyclic (co)fibrations, which we now recall.
\begin{itemize}
\item A \emph{core fibration} is a fibration between fibrant objects. Dually, a \emph{core cofibration} is a cofibration between cofibrant objects.
\item A cofibration is said to be \emph{acyclic} if it has the left lifting property against all core fibrations. In particular each trivial cofibration is acyclic.  The acyclic fibrations are defined dually and, as before, we note that each trivial fibration is acyclic.
\end{itemize}
\end{notation}

A relative cylinder object for a cofibration $A \hookrightarrow B$ is a factorisation:

  \[ B \coprod_A B \overset{i}{\hookrightarrow} I_A B \rightarrow B \]

of the codiagonal map, such that $i$ is a cofibration and the composite $B \hookrightarrow B \coprod_A B \hookrightarrow I_A B$ with the first pushout coprojection is an acyclic cofibration.

Dually, a relative path object for a fibration $Y \twoheadrightarrow X$ is a factorisation:

  \[ Y \rightarrow P_X Y \overset{p}{\twoheadrightarrow}  Y \times_X Y   \]

of the diagonal map such that $p$ is a fibration and the composite $P_X Y \twoheadrightarrow Y \times_X Y \twoheadrightarrow Y$ with the first pullback projection is an acyclic fibration.

\begin{defn} A (factorisation\footnote{The word ``factorisation'' refers to the fact that the notion of weak model category in \cite{henry2018weakmodel} is slightly more general. There, the existence of the weak factorisation systems is also weakened. Here we work in the framework of accessible and combinatorial weak model categories from \cite{henry2019CWMS} in which these are always present.}) weak model category is a \wfscat\ which satisfies the following conditions:

  \begin{itemize}

  \item \emph{Cylinder object axiom} --- Each cofibration $A \hookrightarrow B$ with $A$ cofibrant and $B$ fibrant admits a relative cylinder object.

  \item \emph{Path object axiom} --- Each fibration $Y \twoheadrightarrow X$ with $Y$ cofibrant and $X$ fibrant admits a relative path object.

\end{itemize}
\end{defn}

\begin{Remark}\label{rk:weak_cylinder} Because the cylinder and path object axioms are restricted to maps from cofibrant to fibrant objects, they only imply the existence of cylinder objects and path objects for bifibrant objects. Indeed, a cylinder object $X \coprod X \hookrightarrow IX \rightarrow X$ for a cofibrant object $X$ is a relative cylinder object for the cofibration $\emptyset \hookrightarrow X$.  Hence we need $X$ to be both fibrant and cofibrant in order to obtain a cylinder object from the cylinder object axiom as stated.

Although for a general core cofibration $A \hookrightarrow B$ we cannot construct a relative cylinder object, it was observed in \cite{henry2018weakmodel} (see Definition~2.1.12 and Remark~2.1.13) that we can construct a weakening of it called a ``weak relative cylinder object'', which is still sufficient to define the homotopy relation.  A weak relative cylinder object is a diagram of the form:

\[\begin{tikzcd}
  B \coprod_A B \ar[r,hook] \ar[d] & I_A B \ar[d] \\
 B \ar[r,hook,"\sim"] & D_A B  
\end{tikzcd}\]

in which the first of the two maps $B \hookrightarrow I_A B$ is an acyclic cofibration.  For instance $D_A B$ might be a fibrant replacement of $B$.   Note that if $B$ is fibrant then the acyclic cofibration $B \overset{\sim}{\hookrightarrow} D_A B$ has a section so that one recovers a ``strong'' cylinder object. 
\end{Remark}

The general theory of weak model categories can be found in \cite{henry2018weakmodel}, but we recommend the reader to start with the preliminary section of \cite{henry2019CWMS} for a shorter overview of the basic theory.  Below we highlight a few of basic facts.
\begin{itemize}

\item Using the path (or, equivalently, cylinder objects) one can define the \emph{homotopy relation} for maps from cofibrant to fibrant objects.  This is an equivalence relation having the usual properties.  In particular, one can construct the homotopy category $Ho(\Ccal)$ of $\Ccal$ by taking homotopy classes of maps between bifibrant objects.

\item One can show that $Ho(\Ccal)$ is equivalent to the localisation of the full subcategory $\Ccal_{c \cup f}$ of $\Ccal$ of objects that are either fibrant or cofibrant at the union of the classes of core acyclic cofibrations and core acyclic fibrations.  A weak equivalence is defined to a morphism that is inverted by $\Ccal_{c \cup f} \to Ho(\Ccal)$.  Of course, the notion of weak equivalence therefore makes sense only for arrows between objects that are either fibrant or cofibrant.

\end{itemize}

\begin{Properties}\label{properties:model}
The present paper will not be self-contained, relying on a number of results from \cite{henry2018weakmodel, henry2019CWMS}.  For the purposes of readability, we now list the various results about weak model categories that we need.  Each of these has a dual version.
\begin{enumerate}[(i)]
\item In a weak model category $\C$ a core cofibration is acyclic if and only if it is a weak equivalence.  This is Proposition 2.2.9(iv) of \cite{henry2018weakmodel}.\label{item:model1}
\item Since weak equivalences satisfy 2-out-of-3, it follows that the acyclic cofibrations satisfy 2-out-of-3 amongst core cofibrations.\label{item:model2}
\item Conversely, Proposition 2.3.3 of \cite{henry2018weakmodel} asserts that if a pre-model category $\C$ satisfies the cylinder object axiom and the acyclic cofibrations between bifibrant objects satisfy the right cancellation property --- given core cofibrations $j:A \hookrightarrow B$ and $i:B \hookrightarrow C$ such that $i$ and $i \circ j$ are acyclic, then $i$ is acyclic too.
--- then $\C$ is a weak model category.\label{item:model3}
\item The third item above is useful in recognising weak model structures, but the condition there on relative cylinder objects can in fact be further relaxed.  Indeed, to construct relative cylinder objects it is enough that: 
\begin{itemize}
\item given a core cofibration $A \hookrightarrow B$ with $B$ fibrant there exists a cofibration $B \coprod_A B \hookrightarrow I_A B$ such that both maps $B \hookrightarrow I_A B$ are acyclic cofibrations.
\end{itemize}
The point here is that there is a ``trick" to construct a weak relative cylinder object for $A \hookrightarrow B$ by considering the pushout $I_A B \coprod_B I_A B$, and from this one can deduce the cylinder object axiom  (See Lemma~2.3.6 and Remark~2.3.7 of \cite{henry2018weakmodel}). \label{item:model4} 
\end{enumerate}
\end{Properties}

\subsection{Quillen equivalences between weak model categories}

A \emph{Quillen adjunction} between pre-model categories is an adjunction:

\[ F: \Ccal \leftrightarrows \Acal : U \]

such that $F$ preserves cofibrations and $U$ preserves fibrations.  By adjointness, these conditions amount to asking that $F$ preserves cofibrations and trivial cofibrations, or equally that $U$ preserves fibrations and trivial fibrations.  Note that since $F$ then preserves core cofibrations and $U$ preserves core fibrations it follows also -- see Lemma 2.22 of \cite{henry2019CWMS} of -- that $F$ preserves acyclic cofibrations and that $U$ preserves acyclic fibrations.

Now, by Proposition 2.4.3 of \cite{henry2018weakmodel}, if the Quillen adjunction \[ F: \Ccal \leftrightarrows \Acal : U \] is between weak model categories
then it induces an adjunction \[ Ho(\Ccal) \leftrightarrows Ho(\Acal) \]
between homotopy categories. 

\begin{defn}
A Quillen adjunction as above is said to be a \emph{Quillen equivalence} just when the adjunction on homotopy categories is an equivalence of categories.
\end{defn}

\begin{Properties}\label{properties:adj}
The following are a few characterisations of Quillen equivalences that we will need later, given in Proposition 2.4.5 of \cite{henry2018weakmodel}.
\begin{enumerate}[(i)]
\item For $X \in \Ccal$ cofibrant and $Y \in \Acal$ fibrant a map $FX \to Y$ is a weak equivalence if and only if its adjoint $X \to UY$ is.\label{item:adj1}
\item At $Y \in \Acal$ fibrant the derived counit $FQUY \to FUY \to Y$ is a weak equivalence where $QUY \to UY$ is a cofibrant replacement and $F$ reflects weak equivalences between cofibrant objects.\label{item:adj2}
\item As in \ref{item:adj2} but asking only that $F$ reflects weak equivalences between bifibrant objects.\label{item:adj3}
\end{enumerate}
\end{Properties}

\subsection{Accessible weak factorisation systems and weak model categories}
A pre-model category (or a weak model category) is said to be accessible (resp. combinatorial) if its underlying category is locally presentable and its factorisation systems are accessible (resp. cofibrantly generated).  The notion of cofibrantly generated weak factorisation system is well known, but let us remind the reader what accessibility means in this context.

To this end, recall that a \emph{functorial factorisation} on a category $\C$ is a
functor 
\begin{equation}\label{eq:ffs}
\xymatrix{
\C^\mathbf 2 \ar[rr]^{(L,E,R)} && \C^{\mathbf 3}: & X \ar[r]^{f} & Y & \longmapsto & X \ar[r]^{Lf} & Ef \ar[r]^{Rf} & Y
}
\end{equation} 
from the category of arrows to that of composable pairs which is a section of the composition functor $\C^\mathbf 3 \to \C^\mathbf 2$.  

We say that a weak factorisation system admits a functorial factorisation if such a functor exists for which the morphisms $Lf$ and $Rf$ always belong to the left and right class of the given weak factorisation system -- furthermore, if $\C$ is locally presentable and the functorial factorisation $(L,E,R):\C^{\atwo} \to \C^{\athree}$ can be chosen so as to be an \emph{accessible functor}, then we say that the weak factorisation system is accessible.

\subsection{Saturation conditions}
A pre-model category is said to be
\begin{itemize}
\item \emph{left saturated} if the class of trivial cofibrations and acyclic cofibrations coincide.
\item \emph{core left saturated} if the classes of trivial cofibrations and acyclic cofibrations between cofibrant objects coincide.
\item \emph{right saturated} or \emph{core right saturated} if the dual conditions hold.
\end{itemize}

In Section 4 of \cite{henry2019CWMS} it is shown that any accessible pre-model category can be made left and right saturated without changing its classes of core cofibrations and core fibrations, so that one can always assume these conditions are satisfied without affecting the homotopy theoretic information.

\begin{Properties}\label{properties:saturated}
\begin{enumerate}[(i)]
\item If every object is fibrant then the class of acyclic and trivial cofibrations coincide --- that is, the pre-model category is left saturated.\label{item:sat1}
\item More generally, each acyclic cofibration with fibrant target is a trivial cofibration -- this is established in Lemma 2.17 of \cite{henry2019CWMS}.\label{item:sat2}

\end{enumerate}
\end{Properties}

\subsection{Left and right semi-model categories as weak model categories}

The original definitions of left and right semi-model category \cite{barwick2010left, spitzweck2001operads} do not require the existence of two fully formed weak factorisation systems.  Since we are only interested in the case where such factorisations do exist, we give the following simplified definition.

\begin{Definition}
A pre-model category $\C$ is a left semi-model category if it can be endowed with a class of arrows $W$ (called weak equivalences) such that:
\begin{itemize}
\item $W$ contains the isomorphisms and satisfies 2-out-of-6.
\item A core cofibration is a trivial cofibration if and only if it is a weak equivalence.
\item A fibration is a trivial fibration if and only if it is a weak equivalence.
\end{itemize}
\end{Definition}

\begin{Remark}\label{remark:data}
In the above formulation, the class of weak equivalences is not specified as part of the data, whilst in \cite{barwick2010left, spitzweck2001operads} it is specified.  However, the above axioms ensure that the class of weak equivalences above is uniquely determined --- it is the 2-out-of-3 closure of the trivial fibrations and core trivial cofibrations --- from which it immediately follows that the two formulations are equivalent.

\end{Remark}

\begin{Properties}\label{properties:semivsweak}
\begin{enumerate}[(i)]
\item A left semi-model category in which each object is cofibrant is a Quillen model category.\label{item:semi1}
\item Each left semi-model category has an underlying weak model category.  Furthermore, Remark~\ref{remark:data} shows that being a left semi-model category is a property of a factorisation weak model category.
\item In fact, Section 3 of \cite{henry2019CWMS} shows that left semi-model categories can be characterised as the factorisation weak model categories which are right saturated, core left saturated and for which each cofibrant object $X$ admits a cylinder object $X \coprod X \hookrightarrow IX \overset{\sim}{\rightarrow} X$.\label{item:semi2}
\end{enumerate}
\end{Properties}

\section{Preliminaries on algebraic weak factorisation systems}\label{sect:algebraic}
We now describe the necessary background on algebraic weak factorisation systems --- see, for instance \cite{Bourke2016Accessible, Garner2011Understanding, Grandis2006Natural, Riehl2011Algebraic}.  This background material is mostly a summary of material in Section 2 of \cite{Bourke2016Accessible}, to which we refer the reader for a more detailed account.  

\subsection{Algebraic weak factorisation systems}

A functorial factorisation $(L,E,R):\C^{\mathbf 2} \to \C^{\mathbf 3}$ as in ~\eqref{eq:ffs} induces endofunctors $L,R:\C^{\mathbf 2} \to \C^{\mathbf 2}$ of the arrow category and natural transformations $\epsilon: L \to 1$ and $\eta:1 \to R$ with components as below.
\begin{equation*}
\xymatrix{
A \ar[r]^1 \ar[d]_{Lf} & A \ar[d]^{f} \\
Ef \ar[r]^{Rf} & B
} \hspace{2cm}
\xymatrix{
A \ar[r]^{Lf} \ar[d]_{f} & Ef \ar[d]^{Rf} \\
B \ar[r]^{1} & B\rlap{ .}
}
\end{equation*}
An algebraic weak factorisation system consists of a comonad $(L,\epsilon,\Delta)$ and monad $(R,\eta,\mu)$ extending the above, together with a distributive law relating $L$ and $R$ whose details will not be required.

Associated to the algebraic weak factorisation system $(L,R)$ we have the categories $\Coalg{L}$ and $\Alg{R}$ of coalgebras for the comonad $L$ and algebras for the monad $R$ -- these are thought of as the left and right classes of the \awfs.  We denote a coalgebra by $\alg f = (f, s) \colon A \to B$ where $f:A \to B$ is the underlying morphism in $\C$ and $s:f \to Lf$ the structure map for the coalgebra; likewise, we write $\alg g = (g, r) \colon A \to B$ for an $R$-algebra where $r:Rg \to g$ is the structure map.

\subsection{Double categorical aspects}\label{sect:double}
A morphism in $\Coalg{L}$ is a coalgebra map -- that is, a commutative square as below left which is preserves the $L$-coalgebra structure.
\begin{equation*}
\xymatrix{
A \ar[d]_{\alg f = (f,s)} \ar[r]^{a} & C \ar[d]^{\alg g = (g,r)} \\
C \ar[r]^{b} & D}
\end{equation*}

Composition in $\Coalg{L}$ witnesses that these squares can be composed, across the page.  There is also a vertical composition operation -- given $L$-coalgebras $\alg f:A \to B$ and $\alg g:B \to C$ the composite morphism $g \circ f$ in $\C$ admits an extension to a coalgebra $\alg g \cdot \alg f:A \to C$.  This composition is associative and unital when the identity $1_A:A \to A$ is equipped with its unique coalgebra structure $\alg{1_A}:A \to A$.  This \emph{vertical composition} respects $L$-coalgebra homomorphisms and, in total, makes $\Coalg{L}$ into a \emph{double category} whose objects and horizontal morphisms are as in $\C$, whose vertical morphisms are the $L$-coalgebras and whose squares are the $L$-coalgebra morphisms.

\subsection{Algebraic weak factorisation systems versus weak factorisation systems}
If $f$ admits the structure of an $L$-coalgebra and $g$ admits the structure of an $R$-algebra, then $f$ has the left lifting property with respect to $g$.  It follows that each algebraic weak factorisation has an \emph{underlying weak factorisation system}, whose left and right classes consist of the retract closures of the $L$-coalgebras and $R$-algebras respectively.  

In the case that a weak factorisation system underlies an algebraic weak factorisation system, we call the algebraic weak factorisation system an \emph{algebraic realisation} of the weak factorisation system.

\subsection{Accessibility of algebraic weak factorisation systems}
In the case that $\C$ is a locally presentable category the algebraic weak factorisation is said to be \emph{accessible} if its functorial factorisation $(L,E,R)$ is accessible.  In fact, each accessible weak factorisation system admits an accessible algebraic realisation --- see Proposition 3.5 of \cite{garner2020lifting}.

\subsection{Algebraic cofibrant and fibrant replacement}

The counit of the comonad $L$ has component $\epsilon_f = (1_A,Rf):Lf \to f$ -- in particular, has identity domain component $1_A$.  It follows that the comultiplication $\Delta_f$ has the same property whence the comonad restricts to a comonad on each coslice $A/\C$.  If $\C$ has an initial object $\emptyset$,  we thereby obtain a comonad $Q$ on $\emptyset /\C \cong \C$ which is called the \emph{cofibrant replacement comonad}.  Dually, if $\C$ has a terminal object $1$ then the monad $R$ restricts to a \emph{fibrant replacement monad} $T$ on the slice $\C / 1 \cong \C$.  %We will use these in what follows.

\subsection{Cofibrant generation}
A set $J$ of morphisms in a locally presentable category $\C$ cofibrantly (or freely) generates an algebraic weak factorisation system $(L,R)$.  An $R$-algebra consists of a morphism $g:X \to Y$ together with a lifting function $\phi$, which assigns to each commutative square as on the outside left below
\[
\begin{tikzcd}
A \ar[d,"j \in J"swap] \ar[r,"a"] & X \ar[d,"g"] \\
B \ar[r,"b"below] \ar[ur,"{\phi(j,a,b)}"description,dotted] & Y \\  
\end{tikzcd}
\hspace{2cm}
\begin{tikzcd}
A \ar[d,"j \in J"swap] \ar[r,"a"] & X \\
B  \ar[ur,"{\boldsymbol{x}(j,a)}"swap,dotted] \\  
\end{tikzcd}
 \] 
a chosen diagonal filler, as depicted.  The morphisms of $\Alg{R}$ are commutative squares which preserve the chosen fillers.  (Although not required in the present paper, we note that one can replace $J$ in the above by a small category or even double category of morphisms, and each accessible algebraic weak factorisation system is cofibrantly generated in this stronger sense.)

The algebraically fibrant objects -- that is, the $T$-algebras -- are algebraically $J$-fibrant objects.  These consist of an object $X \in \C$ together with a lifting function $\boldsymbol{x}$, which assigns to the solid part of each diagram as on the right above
 a filler.  Again morphisms of $\Alg{T}$ preserve the fillers.   Of course, the category of algebraically $J$-fibrant objects makes sense for any set $J$ and category $\C$, regardless of whether the cofibrantly generated \awfs exists, and we denote it by $\JFib$.  In particular when $(L,R)$ does exist, we have $\Alg{T} \cong \JFib$.

\section{The equivalence transfer theorem}\label{sect:transfer}

\begin{thm} \label{thm:main_transfert} Let $\Ccal$ and $\Acal$ be two \wfscats\ together with a Quillen adjunction:

\[ F: \Ccal \leftrightarrows \Acal : U \]

and suppose that $\Acal$ is a weak model category. Then $\Ccal$ is a weak model category such that the above adjunction is a Quillen equivalence if and only if the following two conditions are satisfied:

\begin{enumerate}[(i)]

\item Given a cofibrant object $X \in \Ccal$, a fibrant object $Y \in \Acal$ and a morphism:

\[ F X \overset{i}{\rightarrow} Y \]

there exists a cofibration $j:X \hookrightarrow Z$ in $\Ccal$ and a factorisation:

\[
\begin{tikzcd}
  F X \ar[r,"i"] \ar[d,hook,"F j" swap] & Y \\
 F Z \ar[ur,"w","\sim" swap]
\end{tikzcd}
\]

such that $w$ is a weak equivalence.

\item If $i$ is a core cofibration in $\Ccal$ and $Fi$ is an acyclic cofibration, then $i$ is an acyclic cofibration.

\end{enumerate}

\end{thm}

\begin{Remark} Our main application of this theorem is when $\Acal$ is an accessible weak model category and $\Ccal$ is the category of algebraically cofibrant objects therein. The dual version of the theorem will be used for the category of algebraically fibrant objects.
\end{Remark}

We first prove that if $\Ccal$ is a weak model category and $(F,U)$ is a Quillen equivalence then these conditions are satisfied. The other implication will occupy us for the remainder of the section.

\begin{proof}
By Properties~\ref{properties:adj}\ref{item:adj2} the left Quillen equivalence $F$ detects weak equivalences between cofibrant objects.  Therefore if $i$ is a core cofibration such that $Fi$ is an acyclic cofibration (and in particular a weak equivalence) it follows that $i$ is a weak equivalence, and so an acyclic cofibration as in Properties~\ref{properties:model}\ref{item:model2}.  This proves the second point immediately.

Consider now a map $F X \rightarrow Y$ with $X$ cofibrant and $Y$ fibrant. It corresponds to a morphism $ X \rightarrow U Y$, which we can factor as a cofibration followed by a trivial fibration:

 \[ X \hookrightarrow Z \overset{\approx}{\twoheadrightarrow} U Y \]

Since the trivial fibration $Z \overset{\approx}{\twoheadrightarrow} UY$ is an acyclic fibration between fibrant objects and $\Ccal$ is a weak model category, it is a weak equivalence by the dual of Properties~\ref{properties:model}\ref{item:model1}.  As $(F,U)$ is a Quillen equivalence, $Z$ is cofibrant and $Y$ is fibrant it follows from Properties~\ref{properties:adj}\ref{item:adj1} that its adjoint map $ F Z \rightarrow Y$ is a weak equivalence in $\Acal$, which gives the desired factorisation. 
\end{proof}

\begin{Assumptions}
For the remainder of the section, we assume that we have a Quillen adjunction $F: \Ccal \leftrightarrows \Acal : U$ satisfying the two conditions of the theorem, and we will show that $\Ccal$ is a weak model structure and that $(F,U)$ is a Quillen equivalence.
\end{Assumptions}

\begin{Lemma} Under the assumption above, $\Ccal$ is also a weak model category. \end{Lemma}

\begin{proof}
Firstly, suppose that in $\Ccal$ one has core cofibrations $A \overset{j}{\hookrightarrow} B \overset{i}{\hookrightarrow} C $ such that $j$ and $i \circ j$ are acyclic.  Then since $F$ is left Quillen both $Fj$ and $F(i \circ j)$ are acyclic cofibrations so that, since $\Acal$ is a weak model category, $Fi$ is an acyclic cofibration too.  Therefore, by the second condition of Theorem~\ref{thm:main_transfert}, $i$ is acyclic too.

In order to show that $\Ccal$ is a weak model category, it suffices by Properties~\ref{properties:model}\ref{item:model4} to show that given a core cofibration $A \hookrightarrow B$, one can construct a cofibration $B \coprod_A B \hookrightarrow I_A B$ such that both maps $B \hookrightarrow I_A B$ are acyclic.

To this end, consider the codiagonal map $B \coprod_A B \rightarrow B$, and let $FB \overset{\approx}{\hookrightarrow} Y \twoheadrightarrow 1$ be a fibrant replacement of $FB$ in $\Acal$.  Applying the first condition of Theorem~\ref{thm:main_transfert} to the map $F(B \coprod_A B) \rightarrow Y$ we obtain a cofibration $B \coprod_A B \hookrightarrow I_A B$ in $\Ccal$ and a weak equivalence $F(I_A B) \rightarrow Y$ in $\Acal$ factoring the map above. In particular, by $2$-out-of-$3$ the maps $F(B) \rightarrow F(I_A B)$ are weak equivalences, hence acyclic cofibrations, so that by the second condition of Theorem~\ref{thm:main_transfert} the maps $B \hookrightarrow I_A B$ are acyclic cofibrations.

\end{proof}

\begin{Lemma} Still under the same assumptions, $(F,U)$ is a Quillen equivalence between $\Ccal$ and $\Acal$. \end{Lemma}

\begin{proof}
Let $Y$ be a fibrant object in $\Acal$ and $u:X \overset{\approx}{\twoheadrightarrow} U Y$ be a cofibrant replacement of $UY$ in $\Ccal$.  The derived counit of the adjunction $F \dashv U$ at $Y$ is then given by the adjoint map $\overline{u}:FX \to Y$ -- we first show this to be a weak equivalence in $\Acal$.  Indeed using the first condition of Theorem~\ref{thm:main_transfert} one can factor it as:

\[ FX \overset{F i}{ \hookrightarrow } F Z \underset{\sim}{\overset{w}{\rightarrow}} Y. \]

By adjointness we obtain a commutative diagram in $\Ccal$ as on the outside of the square below

\[\begin{tikzcd}
X \ar[d,hook,"i"] \ar[r,equal] & X \ar[d,two heads,"\overline{u}"] \\
Z \ar[ur,dotted,"t"] \ar[r,"w^*"] & U Y
\end{tikzcd}
\]

where $w^*$ denotes the adjoint transpose of $w$.  Since $i$ is a cofibration and $u$ a trivial fibration we obtain a diagonal filler $t$ as above.  Thus $u$ fits into a retract diagram as below left

\[
\begin{tikzcd}
  X \ar[d,"u"] \ar[r, "i"] & Z \ar[d,"w^*"] \ar[r, "t"] & X \ar[d,"u"] \\
U Y \ar[r,equal] & U Y \ar[r,equal] & U Y \\
\end{tikzcd}
\hspace{1.5cm}
\begin{tikzcd}
 F X \ar[d,"\overline{u}"] \ar[r, "Fi"] & FZ \ar[d,"w"] \ar[r, "F t"] & FX \ar[d,"\overline{u}"] \\
Y \ar[r,equal] & Y \ar[r,equal] & Y \\
\end{tikzcd}
\]

so that, by adjointness, we obtain a retract diagram as above right. Then $\overline{u}$ is a retract of the weak equivalence $w$, and so a weak equivalence itself.

To conclude the proof it suffices, by Properties~\ref{properties:adj}\ref{item:adj3}, to show that $F$ detects weak equivalences between bifibrant objects. Thus let $f: X \rightarrow Y$ be a map between bifibrant objects such that $F f$ is a weak equivalence, and consider a (cofibration/trivial fibration)-factorisation
\[ X \overset{i}{ \hookrightarrow } Z \underset{\sim}{\overset{w}{\rightarrow}} Y \]
of $f$. Then $w$ a weak equivalence between cofibrant objects and so is sent by $F$ to a weak equivalence, so that $Fi$ is also a weak equivalence by $2$-out-of-$3$, and hence is an acyclic cofibration. It therefore follows by our assumption on $F$ that $i$ is an acyclic cofibration between cofibrant objects and so a weak equivalence. Hence $f = w \circ i$ is a weak equivalence too.

\end{proof}

\begin{Remark}

In practice, one typically starts with a weak model structure on $\Acal$, and an adjunction $F:\Ccal \leftrightarrows \Acal:U$ which satisfies the first condition of Theorem~\ref{thm:main_transfert} but with ``a cofibration in $\Acal$'' replaced by ``a map $i$ in $\Acal$ such that $Fi$ is a cofibration''.  Given this, one then forms the left transferred weak factorisation systems on $\Acal$ --- that is, such that the (trivial) cofibrations in $\Acal$ are the preimage under $F$ of the (trivial) cofibrations in $\Ccal$. Unfortunately, in order to apply Theorem~\ref{thm:main_transfert} one requires that a map in $\Acal$ is an acyclic cofibration if and only if $Fi$ is an acyclic cofibration, and this does not appear to be automatic in the situation above. 

To avoid this problem there are two solutions.  Firstly, one can assume that $\Acal$ is core left saturated --- that is, all core acyclic cofibration in $\Acal$ are trivial cofibration --- or replace $\Acal$ by its core left saturation  in which this property holds --- see Section 4 of \cite{henry2019CWMS}. Alternatively, one can use the following modification of Theorem~\ref{thm:main_transfert} in which Condition 2 is restricted to trivial cofibrations, but Condition 1 is strengthened with the assumption that $FZ$ is fibrant. In the case of the weak model structure on algebraically cofibrant objects, this is ensured by the fact that $F Z \rightarrow Y $ is constructed as a trivial fibration.
\end{Remark}

\begin{Prop} \label{prop:main_transfert_Simplified}

Let $\Ccal$ and $\Acal$ be two \wfscats\ with a Quillen adjunction:

\[ F: \Ccal \leftrightarrows \Acal : U \]

Suppose that $\Acal$ is a weak model category and that the adjunction satisfies the following two conditions:

\begin{enumerate}[(i)]

\item Given a cofibrant object $X \in \Ccal$, a fibrant object $Y \in \Acal$ and a morphism:

\[ F X \overset{i}{\rightarrow} Y \]

there exists a cofibration $j:X \hookrightarrow Z$ in $\Ccal$ and a factorisation:

\[
\begin{tikzcd}
  F X \ar[r,"i"] \ar[d,hook,"F j" swap] & Y \\
 F Z \ar[ur,"w","\sim" swap]
\end{tikzcd}
\]

such that $w$ is a weak equivalence and $F Z$ is fibrant in $\Acal$.\label{item:simplified1}

\item If $i$ is a core cofibration in $\Ccal$ and $Fi$ is a trivial cofibration, then $i$ is an acyclic cofibration.\label{item:simplified2}

\end{enumerate}
  
Then $\Ccal$ is a weak model category and $(F,U)$ is a Quillen equivalence.

\end{Prop}

\begin{proof}

By Theorem~\ref{thm:main_transfert} we need only show that if $i : X \hookrightarrow Y$ is a core cofibration in $\Ccal$ such that $Fi$ is acyclic then $i$ is acyclic. To this end, consider a fibrant replacement:
\[ FX \hookrightarrow FY \overset{\approx}{\hookrightarrow} Z \]
By Condition~\ref{item:simplified1} there is a cofibration $Y \hookrightarrow Z'$ in $\Ccal$ and factorisation
\[ FY \hookrightarrow FZ' \overset{\approx} \rightarrow Z \]
of the map to the fibrant replacement such that $FZ'$ is fibrant and such that the map $FZ' \to Z$ is a weak equivalence. In particular the composite $X \hookrightarrow Y \hookrightarrow Z'$ is a core cofibration whose image by $F$ is an acyclic cofibration with fibrant target.  Hence its image under $F$ is a trivial cofibration (see Properties~\ref{properties:saturated}\ref{item:sat2}), so the map $X \hookrightarrow Z'$ is acyclic in $\Ccal$ by assumption.  A similar argument shows that the map $Y \hookrightarrow Z'$ is also acyclic in $\Ccal$.  Hence, by the elimination property of acyclic cofibrations (Properties~\ref{properties:model}.\ref{item:model2}), the map $X \hookrightarrow Y$ is an acyclic cofibration too.
\end{proof}

We also state explicitly the dual version of Theorem~\ref{thm:main_transfert}:

\begin{thm} \label{thm:main_dual_transfert} Let $\Ccal$ and $\Acal$ be two \wfscats~together with a Quillen adjunction:

\[ F: \Ccal \leftrightarrows \Acal : U \]

Assume that $\Ccal$ is a weak model category. Then $\Acal$ is a weak model category and the adjunction above is a Quillen equivalence if and only if the following two conditions are satisfied:

\begin{enumerate}[(i)]

\item Given a cofibrant object $X \in \Ccal$, a fibrant object $Y \in \Acal$ and a morphism:

\[ X \overset{i}{\rightarrow} U Y \]

there exists a fibration $p:Z \twoheadrightarrow Y$ in $\Acal$ and a factorisation:

\[
\begin{tikzcd}
& U Z \ar[d,->>,"U p"] \\
  X \ar[ur,"w","\sim"swap] \ar[r,"i"swap]  & U Y \\
\end{tikzcd}
\]

such that $w$ is a weak equivalence in $\Ccal$.\label{item:main1}

\item If $i$ is a core fibration in $\Acal$ and $Ui$ is an acyclic fibration in $\Ccal$, then $i$ is an acyclic fibration in $\Acal$.\label{item:main2}

\end{enumerate}

\end{thm}

It is proved by applying Theorem~\ref{thm:main_transfert} to the Quillen adjunction:

 \[ U^{op} : \Acal^{op} \leftrightarrows \Ccal^{op} : F^{op}. \]

\begin{Remark} Dualizing Proposition~\ref{prop:main_transfert_Simplified}, if we can further show that in Condition~\ref{item:main1} the object $Z$ can be chosen so that $U Z$ is cofibrant in $\Ccal$, then Condition~\ref{item:main2} can be weakened to the requirement that if $i$ is a core fibration and $Ui$ is a trivial fibration then $i$ is an acyclic fibration. \end{Remark}

\section{Model structures on algebraically cofibrant objects}\label{sect:cofibrant}

\begin{Theorem}\label{th:ModelStr_Cofibrant_object}
Let $\Acal$ be an accessible weak model category endowed with $(L,R)$ an accessible algebraic realisation of the (cofibration, trivial fibration)-weak factorisation system.  Let $Q$ be the associated cofibrant replacement comonad on $\Acal$.

Then there is an accessible weak model structure on $\Coalg{Q}$ such that:
\begin{enumerate}[(i)]
\item A map in $\Coalg{Q}$ is a cofibration or trivial cofibration if and only if its underlying map in $\Acal$ is.\label{item:coalg1}
\item The adjunction:

\[ U : \Coalg{Q} \leftrightarrows \Acal : G \]

where $U$ is the forgetful functor, is a Quillen equivalence.\label{item:coalg2}

\item Each object is cofibrant in $\Coalg{Q}$. \label{item:coalg3}

\item An arrow in $\Coalg{Q}$ is a weak equivalence or an acyclic cofibration if and only its underlying map in $\Acal$ is.\label{item:coalg4}

\item If $\Acal$ is core left saturated then $\Coalg{Q}$ is an accessible right semi-model category.\label{item:coalg5}

\item If $\Acal$ is combinatorial then so is $\Coalg{Q}$.\label{item:coalg6}

\end{enumerate}
\end{Theorem}

\begin{proof}
Since $U$ is an accessible left adjoint between locally presentable categories, there exist -- by Theorem $2.6$ of \cite{garner2020lifting} -- accessible weak factorisation systems on $\Coalg{Q}$ whose left classes consist of those maps $f$ for which $Uf$ is a cofibration or trival cofibration respectively.   We thus have a Quillen adjunction of pre-model categories for which \ref{item:coalg1} holds. To prove that this a weak model structure on $\Coalg{Q}$ satisfying \ref{item:coalg2} above, it therefore suffices to verify Condition~\ref{item:simplified1} of Proposition~\ref{prop:main_transfert_Simplified}.  

To this end, let $(X,\boldsymbol{x})$ be a cofibrant $Q$-coalgebra -- such is specified by a cofibrant object $X \in \Acal$ together with an $L$-coalgebra structure $\boldsymbol{x}:\emptyset \to X$.  As in Section~\ref{sect:algebraic} we use boldface to indicate that $\boldsymbol{x}$ is a morphism $x:\emptyset \to X$ equipped with additional structure.  Consider a morphism $f:X \to Y$ with $Y$ fibrant in $\Acal$.  We factor $f$ using the $(L,R)$-algebraic weak factorisation system on $\Acal$ as below
\[ X \overset{Lf}{\rightarrow} Ef \overset{Rf}{\twoheadrightarrow} Y \] 

so that $Lf$ is a cofibration between cofibrant objects and $Rf$ an acylic fibration between fibrant objects. Moreover $Lf$ obtains the structure of a (co-free) $L$-coalgebra $\mathbf {Lf}:X \to Ef$.
We then consider the composite
\begin{equation*}
\xymatrix{
\emptyset \ar[d]_{\boldsymbol{x}} \ar[r]^{1} & \emptyset \ar[d]^{ \boldsymbol{x}} \\
X \ar[d]_{\mathbf{1}} \ar[r]^{1} & X \ar[d]^{\mathbf{\mathbf {Lf}}} \\
X \ar[r]_{Lf} & Ef
}
\end{equation*}
in which the upper square is the identity $L$-coalgebra map.  The lower square is an $L$-coalgebra morphism by the coalgebra dual of Proposition 5 of \cite{Bourke2016Accessible}.  Now since $\Coalg{L}$ is, as described in Section~\ref{sect:double}, a double category, the composite square is a morphism of $L$-coalgebras $\boldsymbol{x} \to \mathbf{Lf} \circ \boldsymbol{x}$.  This equips the fibrant object $Ef$ with the structure of a $Q$-coalgebra $(Ef, \mathbf{Lf} \circ \boldsymbol{x})$ such that $Lf:(X,\boldsymbol{x}) \to (Ef, \mathbf{Lf} \circ \boldsymbol{x})$ is a morphism of $Q$-coalgebras. In particular, it is a cofibration in $\Coalg{Q}$ whose image under $U$ is $Lf$, therefore verifying Condition \ref{item:simplified1} of Proposition~\ref{prop:main_transfert_Simplified}, as required.

For \ref{item:coalg3} observe that, by definition of $Q$, for each $Q$-coalgebra $X$, the map $!:\emptyset \to X$ is an $L$-coalgebra, hence a cofibration.

Then for \ref{item:coalg4} since $U$ is a left Quillen equivalence it preserves and detects weak equivalences between cofibrant objects.  Since all objects of $\Coalg{Q}$ are cofibrant an arrow $f \in \Coalg{Q}$ is a weak equivalence if and only if $Uf \in \Acal$ is a weak equivalence.  Since $U$ is left Quillen it preserves acyclic cofibrations.  Conversely if $Uf$ an acyclic cofibration between cofibrant objects, then $Uf$ is both a cofibration and a weak equivalence in $\Acal$.  Hence, by the above, $f$ is both a cofibration and weak equivalence in $\Coalg{Q}$.  Since each object of $\Coalg{Q}$ is cofibrant $f$ is furthermore a core cofibration and so an acyclic cofibration by Properties~\ref{properties:model}\ref{item:model1}.

For \ref{item:coalg5}, as every object in $\Coalg{Q}$ is cofibrant, all fibrant objects admit path objects. Hence Properties~\ref{properties:semivsweak}.\ref{item:semi2} implies that it is a right semi-model category if and only if it is left saturated and core right saturated.  As every object of $\Coalg{Q}$ is cofibrant it is automatically right saturated (see Properties~\ref{properties:saturated}\ref{item:sat1}). Furthermore \ref{item:coalg4} shows that if $\Acal$ is core left saturated then $\Coalg{Q}$ is left saturated: indeed, given an acyclic cofibration $j$ in $\Coalg{Q}$, $Uj$ is an acyclic core cofibration in $\Acal$, so that if $\Acal$ is core left saturated then it is a trivial cofibration and so $j$ is a trivial cofibration too.

Condition \ref{item:coalg6} holds by Remark 3.8 in \cite{makkai2014cellular}, which proves that combinatorial weak factorisation systems between locally presentable categories can be left transferred along a left adjoint and, moreover, the transferred weak factorisation system is again combinatorial.
\end{proof}

\begin{Remark}\label{rk:ModelStr_Cofibrant_object_more_gen}
Analysing the above, one sees that Theorem~\ref{th:ModelStr_Cofibrant_object} applies more generally to any accessible algebraic weak factorisation system $(L,R)$ on $\Acal$ satisfying the following two conditions. 
\begin{itemize}\label{eq:cond}
\item[(1)] If $f:X \to Y$ is an arrow from a cofibrant to a fibrant object, then it can be factored as a cofibration with an $L$-coalgebra structure followed by an acyclic fibration.
\item[(2)] If $\emptyset \to X$ admits an $L$-coalgebra structure then $X$ is cofibrant.
\end{itemize}

Moreover, if just (1) holds then we can still prove Parts~\ref{item:coalg1},\ref{item:coalg2}, \ref{item:coalg6} of Theorem~\ref{th:ModelStr_Cofibrant_object} and a restricted version of \ref{item:coalg4}:

\begin{itemize}
\item[$(iv')$] An arrow between two objects of $\Coalg{Q}$ whose underlying object are cofibrants in $\Ccal$ is a weak equivalence or an acyclic cofibration if and only its underlying map in $\Acal$ is.
\end{itemize}

Condition (2) is required only for Parts~\ref{item:coalg3} and \ref{item:coalg5} and the full version of \ref{item:coalg4}. 
\end{Remark}

\begin{Remark}\label{rk:endofunctor} We can also formulate a version of Theorem~\ref{th:ModelStr_Cofibrant_object} that does not involve an algebraic weak factorisation system but merely an accessible weak factorisation system, with a chosen functorial factorisation $(L,R)$.  In this case $L$ is not a comonad but merely a copointed endofunctor, whilst $R$ is a pointed endofunctor.  Nonetheless we can consider the $L$-coalgebras (for the copointed endofunctor) and $R$-algebras respectively and we require either that the underlying weak factorisation system of $(L,R)$ is the (cofibration,trivial fibration) weak factorisation system, or, more generally, satisfies the analogues of Conditions (1) and (2) of Remark~\ref{rk:ModelStr_Cofibrant_object_more_gen}. Restricting $L$ to the coslice under the initial object, we obtain the copointed endofunctor $Q$ as before and the accessibility assumption ensures that the category $\Coalg{Q}$ of coalgebras for the copointed endofunctor is locally presentable and that $U:\Coalg{Q} \to \Acal$ has a left adjoint.  We claim that the conclusions of Theorem~\ref{th:ModelStr_Cofibrant_object} hold in this modified setting, with the same dependencies upon Conditions (1) and (2) as in the algebraic case.

Indeed, one follows the proof of Theorem~\ref{th:ModelStr_Cofibrant_object} with one exception: given $(X,\boldsymbol{x}) $ a cofibrant $Q$-coalgebra and $X \to Z \to Y$ a factorisation as a cofibration with $L$-coalgebra structure followed by an acyclic fibration, the argument of Lemma 2.12 of \cite{garner2020lifting} allows us to put an $L$-coalgebra structure on $\emptyset \to Z$ that makes $X \to Z$ a morphism (and hence a cofibration) of $Q$-coalgebras.
\end{Remark} 

\begin{Example}\label{ex:campbell} 
In the present example we describe a combinatorial model structure on the category of small categories $\Cat$ for which the induced right model structure on the category of algebraically cofibrant objects --- simply the category of graphs --- is not a full Quillen model structure.  This is a modification of an example that we learned from Alexander Campbell, and which will appear in \cite{Campbell2020}.

To begin with, we equip $\Cat$ with the Quillen model structure in which:
\begin{itemize}
\item A functor $f:X \to Y$ is a weak equivalence if and only if the induced function $\pi_0(f) : \pi_0(X) \to \pi_0(Y)$ between sets of path components is a bijection.
\item The cofibrations are generated by the two inclusions $J=\{ \emptyset \to \bullet, (\bullet \quad \bullet) \rightarrow (\bullet \to \bullet)\}$.  In particular the trivial fibrations are functors that are full and surjective on objects.
\end{itemize}
In order to verify that this is a Quillen model structure, let us check that the two classes satisfy the conditions of Smith's theorem, Theorem 1.7 of \cite{Beke2000Sheafifiable}.  Certainly each trivial fibration is essentially surjective on objects, and so a weak equivalence.  The latter class, being the preimage under $\Pi_0:\Cat \to \Set$ of the isomorphisms,  satisfies 2-out-of-3 and is retract stable.  Since $\Pi_0$ is cocontinuous, having right adjoint the discrete category functor, it follows that the weak equivalences are stable under pushout and, moreover, accessibly embedded in the category of morphisms $\Cat^{\atwo}$.  Since cofibrations are always pushout stable it follows that the intersection of the cofibrations and weak equivalences is also pushout stable.  Therefore the conditions of Smith's theorem are met --- the model structure exists.

Next consider the forgetful functor $U:\Cat \to \Gph$ to the category of graphs and $F$ its left adjoint with $\eta:1 \to UF$ and $\epsilon:FU \to 1$ denoting the unit and counit.  We will prove that $FU$ is the algebraic cofibrant replacement comonad on $\Cat$ induced by the set of generating cofibrations $J$.  Letting $(L,R)$ denote the cofibrantly generated algebraic weak factorisation system an $R$-map is specified by a functor $f:Y \to Z$ equipped with (1) for each $z  \in Z$ an object $\phi_z \in Y$ with $f \phi_z = z$ and (2) for each $(y,\alpha:fy \to fz,z)$ a morphism $\phi_\alpha:y \to z$ such that $f \phi_\alpha = \alpha$.  Observe that if $f$ is bijective on objects, this amounts to giving a section of $Uf$ in $\Gph$.  In particular, since $\epsilon:FUX \to X$ is the identity on objects the unit component $\eta_{UX}:UX \to UFUX$ equips it with the structure of an $R$-map.  It is easy to say that this is the free $R$-map on $\emptyset \to X$, in the sense that if $f:Y \to Z$ is an $R$-map and $g:Y \to Z$ arbitrary, then there exists a unique morphism $\overline{g}:FUX \to Y$ making the square below left
\begin{equation*}
\xymatrix{
FUX \ar[d]_{\epsilon_{X}} \ar[r]^{\overline{g}} & Y \ar[d]^{f} \\
X \ar[r]^{g} & Z
}
\hspace{1.5cm}
\xymatrix{
FUX \ar[d]_{\epsilon_{X}} \ar[r]^{\Delta_{X}} & FUFUX \ar[d]^{\epsilon_{X} \circ \epsilon_{FUX}} \\
X \ar[r]^{1} & X}
\end{equation*}
into a morphism of $R$-maps.  It then follows, by definition of the algebraic cofibrant replacement comonad, that its underlying copointed endofunctor is $\epsilon:FU \to 1$.  By Equation 2.11 of \cite{Bourke2016Accessible} the comultiplication $\Delta_{X}:FUX \to FUFUX$ has the property that the square above right is an $R$-map morphism where the right hand side uses the composition of $R$-maps.  In our setting, this commutativity amounts to commutativity with sections at the level of underlying graphs  --- that is, to
the equation
$$
U\Delta_{X} \circ \eta_{UX} = \eta_{UFUX} \circ \eta_{UX}.
$$
Now this equation is satisfied by $\Delta_{X} = F\eta_{UX}$ and, by the universal property of $FUX$, by $F\eta_{UX}$ alone.  Therefore the cofibrant replacement comonad is given by $(FU,\epsilon, F\eta_U)$ as claimed.

Now since the monad $T=UF$ is parametric right adjoint --- see, for instance Example 2.5 of \cite{Weber2007Familial} --- it preserves all connected limits.  Since $U$ creates limits it follows that $F$ preserves connected limits --- in particular, coreflexive equalisers.  As is well known the monad $T$ is also cartesian so that, in particular, its unit $1 \to UF$ is a cartesian natural transformation.  Since the pullback of an isomorphism is an isomorphism, it follows that $F$ reflects isomorphisms.  Therefore, by the comonad dual of Beck's monadicity theorem, the adjunction $F \dashv U$ is comonadic and so $\Gph \simeq \Coalg{FU}$ is the category of algebraically cofibrant objects.

Applying Theorem~\ref{th:ModelStr_Cofibrant_object} we left transfer the model structure on $\Cat$ to a right semi-model structure on $\Gph$ in which all objects are cofibrant.  Let us examine, firstly, its cofibrations and weak equivalences.  Let $f:X \to Y \in \Gph$.  To say that $Ff:FX \to FY$ induces a bijection on path components is equally to say that $f:X \to Y$ induces a bijection on path components of the graphs.  By the simple nature of the generating cofibrations $J$ in $\Cat$ it is also straightforward to see that $Ff:FX \to FY$ is a cofibration just when $f:X \to Y$ is a monomorphism --- in fact, the $J$-cellular maps $g:FX \to FY$ are precisely those of the form $g=Ff$ for $f:X \to Y$ a monomorphism.  

The right semi model structure is not a Quillen model structure.  To see this, let $\bullet$ denote the graph with a single object but no edges.  Indeed, the codiagonal $\bullet \coprod \bullet \to \bullet$ is not a weak equivalence since it has more than more path component.  However, it is a trivial fibration.  For in any square of the form:

\[ \begin{tikzcd}
  X \ar[d,hook,"i"] \ar[r] & \bullet \coprod \bullet \ar[d] \\
 Y \ar[r] & \bullet 
\end{tikzcd} \]

neither graphs $X$ and $Y$ can have an edge --- hence $i$ is just an injection of sets, so that there is a diagonal filler as in the category of sets, where injections have the left lifting property with respect to surjections.  This right semi-model structure on $\Gph$ is in fact that considered in Section 4.1 of \cite{henry2018weakmodel}, whose fibrant objects are the so-called setoids.

\end{Example}

\begin{Theorem}\label{thm:create}
Under the assumptions of Theorem~\ref{th:ModelStr_Cofibrant_object} suppose further that $\Acal$ is a left semi-model category and that cylinder objects lift along $U:\Coalg{Q} \to \Acal$ --- in the sense that if $(X,\boldsymbol{x})$ is a $Q$-coalgebra, then there exists a cylinder object factorisation $X \coprod X \to IX \to X \in \Acal$ that lifts to a factorisation of $(X,\boldsymbol{x}) \coprod (X,\boldsymbol{x}) \to (X,\boldsymbol{x}) \in \Acal$.  Then $\Coalg{Q}$ is a Quillen model category.
\end{Theorem}

\begin{proof}
As $\Acal$ is a left semi-model category, its core acyclic cofibrations are trivial cofibrations.  Now by Theorem~\ref{th:ModelStr_Cofibrant_object} $U$ reflects trivial and acyclic cofibrations.  Since all objects are cofibrant in $\Coalg{Q}$ it follows that it is left saturated -- that is, its acyclic and trivial cofibrations coincide.  Since all objects are cofibrant it is automatically right saturated by Properties~\ref{properties:saturated}\ref{item:sat1}.  Furthermore, since cylinder objects lift along $U$, $\Coalg{Q}$ has cylinder objects for all cofibrant objects.  Therefore by Properties~\ref{properties:semivsweak}\ref{item:semi2} $\Coalg{Q}$ is a left semi-model category.  But a left semi-model category in which all objects are cofibrant is a Quillen model category, so this concludes the proof.
\end{proof}

In the main theorem of \cite{ching2014coalgebraic}, the authors use simplicial enrichment to lift the model structure to algebraically cofibrant objects.  Below we give a new proof of this result using liftings of cylinder objects, as opposed to a bar construction.  To state the result, recall from \cite{Blumberg2014} and \cite{ching2014coalgebraic} that if $\Acal$ is a combinatorial simplicially enriched model category, then one can run the \emph{enriched} version of the algebraic small object argument to construct a simplicially enriched cofibrant replacement comonad $Q$.

\begin{Theorem}[Riehl-Ching]
Let $\Acal$ be a combinatorial simplicial model category and $Q$ the simplicially enriched cofibrant replacement comonad on $\Acal$.  Then the transferred weak model structure on $\Coalg{Q}$ is in fact a simplicial Quillen model structure and the Quillen equivalence with $\Acal$ an enriched Quillen equivalence.
\end{Theorem}

\begin{proof}
Since $\Acal$ is simplicial and $Q$ a simplicially enriched comonad the category of $Q$-coalgebras is a simplicially enriched category and the adjunction $\Coalg{Q} \leftrightarrows \Acal$ a simplicially enriched adjunction.  It remains, by Theorem~\ref{thm:create} to check that cylinder objects lift along $U$.  Now since $U$ creates colimits, $\Coalg{Q}$ is cocomplete as a simplicially enriched category, and so has tensors by simplicial sets and these are preserved by $U$. Now since $\Acal$ is a simplicial model category, each $X$ admits a cylinder object factorisation
\begin{equation*}
\xymatrix{ X \coprod X \ar[r]^{} & {\Delta[1] \otimes X} \ar[r] &  X.
}
\end{equation*}
given by tensoring $X$ by the standard interval $\Delta[0] \coprod \Delta[0] \to \Delta[1] \to \Delta[0]$.  Since $U$ creates tensors by $\Delta[1]$ and coproducts, if $(X,\boldsymbol{x})$ is a coalgebra, the factorisation lifts to a factorisation of the codiagonal on $(X,\boldsymbol{x})$ as required.
\end{proof}

The above theorem can be generalised as follows.  Suppose that $\Vcal$ is an accessible monoidal model category with cofibrant unit and that $\Acal$, to begin with, is a combinatorial left semi-model $\Vcal$-category --- here, by a left semi-model $\Vcal$-category, we mean one in which the tensoring functor $- . -:\Vcal \times \Acal$ is a left Quillen bifunctor in the usual sense \cite{Hovey1999Model}.\begin{footnote}{In fact, for our purposes it suffices that the two properties for $- . -:\Vcal \times \Acal$ to be a left Quillen bifunctor are required only for morphisms having cofibrant domains and codomains.}\end{footnote} We need, in this setting, to suppose further that for each object $X$ of $\Vcal$ the tensor functor $X \otimes - :\Acal \to \Acal$ preserves cofibrations.  Then --- as described in Corollary 13.2.4 of \cite{Riehl2014Categorical} --- then generating set of cofibrations in $\Acal$ permits the enriched algebraic small object argument.  In particular, one obtains an accessible enriched cofibrant replacement comonad $Q$ on $\Acal$ so that $\Coalg{Q}$ has the structure of a $\Vcal$-category.  Now since the monoidal unit $I$ of $\Vcal$ is cofibrant, then replacing $\Delta[1]$ in the preceding by any cylinder object factorisation $I \coprod I \to J \to I$ allows us to lift cylinder objects to $\Coalg{Q}$.  With the argument otherwise unchanged, we obtain:

\begin{Theorem}
Let $\Vcal$ be an accessible monoidal model category with cofibrant unit and $\Acal$ a combinatorial left semi-model $\Vcal$-category such that for each object $X$ of $\Vcal$ the tensor functor $X \otimes - :\Acal \to \Acal$ preserves cofibrations.  Let $Q$ be the $\Vcal$-enriched cofibrant replacement comonad on $Q$.   Then the transferred weak model structure on $\Coalg{Q}$ is in fact an enriched Quillen model structure and the Quillen equivalence with $\Acal$ an enriched Quillen equivalence.
\end{Theorem}

\section{Model structures on algebraically fibrant objects}\label{sect:fibrant}

Dually to Theorem~\ref{th:ModelStr_Cofibrant_object} one has:

\begin{Theorem}\label{th:ModelStr_fibrant_object}
Let $\Ccal$ be an accessible weak model category endowed with an accessible algebraic realisation $(L,R)$ of the (trivial cofibration, fibration) weak factorisation system, and let $T$ be the corresponding fibrant replacement monad.

Then there is an accessible weak model structure on $\Alg{T}$ such that:

\begin{enumerate}[(i)]
\item\label{item:alg1} A map in $\Alg{T}$ is a fibration or trivial fibration if and only if its underlying map in $\Ccal$ is.
\item\label{item:alg2} The adjunction:

\[ F : \Ccal \leftrightarrows \Alg{T} : U \]

where $U$ is the forgetful functor, is a Quillen equivalence.

\item\label{item:alg3} All objects of $\Alg{T}$ are fibrant.

\item\label{item:alg4} An arrow in $\Alg{T}$ is a weak equivalence or an acyclic fibration if and only its underlying map in $\Ccal$ is.

\item\label{item:alg5} If $\Ccal$ is core right saturated then $\Alg{T}$ is an accessible left semi-model category.

\item\label{item:alg6} If $\Ccal$ is combinatorial then so is $\Alg{T}$.

\end{enumerate}
\end{Theorem}

The proof is the exact dual of the proof of Theorem~\ref{th:ModelStr_Cofibrant_object}, with the only exception that the pre-model structure on $\Alg{T}$ is constructed by right transfer. The existence of right transfer of accessible weak factorisation systems is also proved as Theorem 2.6 of \cite{garner2020lifting}. The existence of right transfer of combinatorial weak factorisation systems is standard and follows directly from the small object argument.

\begin{Remark}\label{rk:generalisation}
 As in the dual Remark~\ref{rk:ModelStr_Cofibrant_object_more_gen}, Theorem~\ref{th:ModelStr_fibrant_object} more generally applies to any accessible algebraic weak factorisation system $(L,R)$ such that:

\begin{enumerate}

\item[(1)] Every arrow from a cofibrant object to a fibrant object can be factored as an acyclic cofibration followed by a fibration admitting $R$-algebra structure.

\item[(2)] Any object such that $X \to 1$ admits $R$-algebra structure is fibrant.
\end{enumerate}

Moreover with only Condition (1) we can already prove Conditions \ref{item:alg1},\ref{item:alg2}, \ref{item:alg6} and the restriction of \ref{item:alg4} to arrows between fibrant objects. As in Remark~\ref{rk:endofunctor}, there is also a variant dealing with an accessible weak factorisation system.

We give an example where these weaker conditions are used below.
\end{Remark}

\begin{example}\label{ex:trivialcof}
Let $\Ccal$ be an accessible weak model category and $J$ a (not necessarily generating) set of acyclic cofibrations. We consider the algebraic weak factorisation system $(L,R)$ cofibrantly generated by $J$, so that $R$-algebras are the $J$-algebraic fibrations and the $T$-algebras are the algebraically $J$-fibrant objects. This algebraic weak factorisation system always satisfies Condition $(1)$: indeed, an arrow from a cofibrant object to a fibrant object can be factored as an acyclic cofibration followed by a (core) fibration, and every core fibration is a $J$-fibration and hence admits an $R$-algebra structure. This is enough to construct a Quillen equivalent weak model structure on the category of algebraically $J$-fibrant objects. Condition (2) holds only if the lifting property against $J$ is enough to characterize fibrant objects. \end{example}

By a dual argument to Theorem~\ref{thm:create} we also have.

\begin{Theorem}\label{thm:create2}
Under the assumptions of Theorem~\ref{th:ModelStr_fibrant_object} suppose further that $\Ccal$ is a right semi-model category and that path objects lift along $U:\Alg{T} \to \Ccal$ --- in the sense that if $(X,\boldsymbol{x})$ is a $T$-algebra, then there exists a cylinder object factorisation $X \to PX \to X \times X \in \Ccal$ that lifts to a factorisation of $(X,\boldsymbol{x}) \to (X,\boldsymbol{x}) \times (X,\boldsymbol{x}) \in \Alg{T}$.  Then $\Alg{T}$ is a Quillen model category.
\end{Theorem}

Up to this point the situation is perfectly symmetric.  However, even if $\C$ is combinatorial the model structure on algebraically cofibrant objects can produce right semi-model structures that are not Quillen model categories (see Example~\ref{ex:campbell}).  On the other hand it is known from \cite{Bourke2019Equipping}, which builds on \cite{Nikolaus2011Algebraic}, that if $\C$ is a combinatorial Quillen model category then the weak model structure on algebraically fibrant objects is again a genuine Quillen model category.  

We generalise this result below to the setting of an algebraic weak factorisation system cofibrantly generated by a set of morphisms.  The key is the following technical lemma, which is an abstraction of the path object obstruction given in Theorem 19 of \cite{Bourke2019Equipping}.
 
 \begin{Lemma} \label{lem:lift_for_cof_gen_AWFS} 
 Let $J$ be a set of morphisms in $\Ccal$ and consider the category $\JFib$ of algebraically $J$-fibrant objects.

 Let $f:X \rightarrow Y$ be an arrow in $\JFib$ which admits a factorisation in $\Ccal$:

\[
\begin{tikzcd}
&Z \ar[dr,->>,"q"] & \\
X \ar[ur,"i"] \ar[rr,"f"] & & Y \\
\end{tikzcd}
\]

such that $i$ is a monomorphism and $q$ is a $J$-fibration. Then $Z$ can be equipped with the structure of an algebraically $J$-fibrant object with respect to which $i$ and $q$ are morphisms of $\JFib$.
  
\end{Lemma}

\begin{proof}
The structure of a algebraically $J$-fibrant object on $Z$ is specified by the choice of a dotted lifting:

\[
\begin{tikzcd}
A \ar[d,"j \in J"swap] \ar[r,"k"] & Z \\
B \ar[ur,dotted,"\boldsymbol{z}(j{,}k)" swap] \\  
\end{tikzcd}
\]

for all $j \in J$ and $k$ as above.  We specify this choice in two stages.
\begin{enumerate}
\item Firstly, if the map $k = i \circ k_0$ factors through $i:X \to Z$, then the factorisation $k_0$ is unique and we define $\boldsymbol{z}(j,k) = i \circ \boldsymbol{x}(j,k_0)$ as below.

\[
\begin{tikzcd}
 A \ar[d,"j \in J"swap] \ar[r,"k_0"] \ar[rr,bend left=30,"k"] & X \ar[r,"i"] & Z \\
 B \ar[ur,"\boldsymbol{x}(j{,}k_0)"description,dotted] \ar[urr,dotted,"\boldsymbol{z}(j{,}k)"swap] \\
\end{tikzcd}
\]
\item On the other hand, if $k$ does not factor through $X$ then firstly we use the $\JFib$ structure on $Y$ to obtain a chosen lift as in the outer square below.
\[
\begin{tikzcd}
A \ar[d,"j \in J"swap] \ar[r,"k"] & Z \ar[d,"q"] \\
B \ar[r,"\boldsymbol{y}(j{,}k)"swap] \ar[ur,dotted] & Y \\  
\end{tikzcd}
 \] 
 \end{enumerate}
Since $q$ is a $J$-fibration there exists a dotted lifting, and we define $\boldsymbol{z}(k,j)$ to be one such lifting -- the particular choice proves to be irrelevant.

It remains to check that for this choice of structure on $Z$, both $i$ and $q$ are morphisms of $\JFib$, which is to say that they preserve the chosen liftings.  Firstly, we consider $i:X \to Z$ and a map $l:A \to X$.  Then $i \circ l:A \to X \to Z$ factors uniquely through $X$ as $l = (i \circ l)_0$ so that the lifting $\boldsymbol{z}(j,i \circ l)$ is constructed according to the first rule above -- the first diagram above interpreted at $k = i \circ l$ then states exactly that $i:X \to Z$ is a morphism of $\JFib$.

For $q$, we need to consider the different two cases.  If $k:A \to Z$ factors through $i:X \to Z$ as $k_0$ the we are in the first situation above, so that $\boldsymbol{z}(j,k)=i \circ \boldsymbol{x}(j,k_0)$.  Then 
$$q \circ \boldsymbol{z}(j,k) = q \circ i \circ \boldsymbol{x}(j,k_0) = f \circ \boldsymbol{x}(j,k_0) = \boldsymbol{y}(j,f \circ k_0) = \boldsymbol{y}(j,q \circ k)$$
where in the second last equation, we use that $f:X \to Z$ is a morphism of $\JFib$.  

In the second case, where $k:A \to Z$ does not factor through $X$, $q$ preserves the lifting $\boldsymbol{y}$ by construction.

\end{proof}

\begin{thm} \label{thm:Alg_Fib_Object_strong_version} Let $\Ccal$ be an accessible right semi-model category with a set $J$ of trivial cofibrations such that every $J$-fibrant object is fibrant. Then there is a Quillen model structure on the category of algebraically $J$-fibrant objects $\JFib$ and a Quillen equivalence

\[ F:\Ccal \leftrightarrows \JFib : U\]

where $U$ is the forgeful functor. The fibrations, weak equivalences, trivial fibrations and acyclic fibrations in $\JFib$ are the maps $f$ such that $Uf$ has the corresponding property in $\Ccal$.

\end{thm}

This improves Theorem 19 of \cite{Bourke2019Equipping} by weakening the hypothesis that $\Ccal$ is a model category to merely a right semi-model category, and we are grateful to Martin Bidlingmaier for a remark that led us to further sharpen the hypotheses on $J$.

\begin{proof}
We consider the algebraic weak factorisation system $(L,R)$ cofibrantly generated by the set $J$. Since each member of $J$ is a trivial cofibration and the morphisms of $J$ suffice to characterise the fibrant objects, by Example~\ref{ex:trivialcof}, we obtain a transferred left semi-model structure on $\JFib$ in which all objects are fibrant.

Now consider an object $(X,\boldsymbol{x})$ of $\JFib$ and the diagonal $(X,\boldsymbol{x}) \to (X,\boldsymbol{x})^{2}$.  Since $\Ccal$ is a right semi-model structure its underlying map admits a path object as below.

\[ X \overset{\delta}{\rightarrow} PX \twoheadrightarrow X \times X \]

Since $\delta$ is a (split) monomorphism and the fibration $PX \twoheadrightarrow X \times X$ belongs to $RLP(J)$, Lemma~\ref{lem:lift_for_cof_gen_AWFS} ensures that the factorisation lifts along $U$ to $\JFib$.  Therefore path objects lift along $U$, so that by Theorem~\ref{thm:create2}, we obtain the full model structure on $\JFib$.

\end{proof}

\begin{Example} It was shown in Section 5.4 of \cite{henry2018weakmodel} that the category of semi-simplicial sets carries a weak model structure such that:

\begin{itemize}
\item The cofibrations are the monomorphisms.
\item The fibrant objects and core fibrations are characterised by their lifting property against horn inclusions.
\item Weak equivalences are the maps whose image in the category of simplicial sets, by the functor that freely adds degeneracies, are weak simplicial equivalences --- equivalently, the maps which induce homotopy equivalences on taking geometric realisations.

\end{itemize}

As every object is cofibrant, it follows that its left core saturation --- see Section~\ref{sect:strictify} below --- is a right semi-model category with the same classes of maps as described above.  Therefore by Theorem~\ref{thm:Alg_Fib_Object_strong_version} there is a Quillen equivalent model structure on the category of ``algebraically fibrant'' semi-simplicial sets, that is, semi-simplicial sets equipped with chosen Kan fillers. The fibrations and weak equivalences are the maps that are fibrations or weak equivalences on underlying semi-simplicial sets. 
\end{Example}

\section{Replacing weak model categories by genuine ones}\label{sect:strictify}

The main goal of this short section is to apply the results established so far to show that each combinatorial weak model category is connected, via a zigzag of Quillen equivalences, to a genuine model category.  We begin by recalling from Section 4 of  \cite{henry2019CWMS} the core left and right saturation of an accessible weak model category.

Let $\Ccal$ be an accessible weak model category.  Then $\Ccal$ admits a second accessible weak model structure called its \emph{core left saturation}, $L^{c}\Ccal$, with the same cofibrations, trivial fibrations and core fibrations as $\Ccal$.  Furthermore $L^{c}\Ccal$ is core left saturated and the identity functor $\Ccal \to L^{c}\Ccal$ is a left Quillen equivalence --- in fact it induces an equivalence between the ordinary categories of bifibrant objects preserving all the relevant structure.

Similarly, there is the core right saturation $R^{c}\Ccal$, for which the identity functor $R^{c}\Ccal \to \Ccal$ is a right Quillen equivalence.  Again this an accessible weak model category.

\begin{Theorem}\label{thm:replace1}
Let $\C$ be an accessible weak model category.  Then $\C$ is connected, via a zigzag of Quillen equivalences, to an accessible right semi-model category in which all objects are cofibrant, and also to an accessible left semi-model category in which all objects are fibrant.
\end{Theorem}
\begin{proof}
In the first case, we consider the zigzag $\Ccal \to L^{c}\Ccal \leftarrow \Coalg{Q}$ in which $Q$ is the cofibrant replacement comonad for an algebraic realisation of the (cofibration,trivial fibration) weak factorisation system on $L^{c}\Ccal$.  In fact, this is the same $Q$ as on $\Ccal$ itself, but the weak model structure might differ.  

The first component is a left Quillen equivalence as mentioned above, and by Theorem~\ref{th:ModelStr_Cofibrant_object} the second component is a left Quillen equivalence.  Moreover since $\L^{c}\Ccal$ is core left saturated, by Theorem~\ref{th:ModelStr_Cofibrant_object} the second component is a right semi-model category with all objects cofibrant.

The proof of the second claim proceeds in the same manner, but now using the core right saturation and Theorem~\ref{th:ModelStr_fibrant_object}.
\end{proof}

\begin{Theorem}
Let $\C$ be a combinatorial weak model category.  Then $\C$ is connected, via a zigzag of Quillen equivalences, to a combinatorial model category in which all objects are fibrant.
\end{Theorem}
\begin{proof}
We consider the zigzag of left Quillen equivalences $\Ccal \to L^{c}\Ccal \leftarrow \Coalg{Q}$ of Theorem~\ref{thm:replace1}.  Since $\Ccal$ is combinatorial, so too is $L^{c}\Ccal$ and $\Coalg{Q}$ respectively.  Then letting $T$ denotes the induced fibrant replacement monad for the combinatorial right semi-model structure $\Coalg{Q}$, it follows from Theorem~\ref{thm:Alg_Fib_Object_strong_version} that $\Alg{T}$ is a combinatorial model category with all objects fibrant and the free functor $\C \to \Alg{T}$ a left Quillen equivalence, so that we have a zigzag of left Quillen equivalences as below
$$\Ccal \to L^{c}\Ccal \leftarrow \Coalg{Q} \to \Alg{T}.$$
\end{proof}

\end{document}